\documentclass[a4paper]{article}

\usepackage[english]{babel}
\usepackage[utf8]{inputenc}
\usepackage{amssymb}
\usepackage{graphicx}
\usepackage[colorinlistoftodos]{todonotes}
\usepackage{amsmath}
\usepackage{amsthm}

\theoremstyle{plain}
\newtheorem{theorem}{Theorem}

\newtheorem{lemma}{Lemma}
\newtheorem*{corollary}{Corollary}

\theoremstyle{definition}
\newtheorem{definition}{Definition}

\title{A Family of Congruences for Rogers--Ramanujan Subpartitions}

\author{Nicolas Allen Smoot \\ Research Institute for Symbolic Computation \\ Johannes Kepler University \\ Linz, Austria}

\date{}

\begin{document}
\maketitle

\begin{abstract}
In 2015 Choi, Kim, and Lovejoy studied a weighted partition function, $A_1(m)$, which counted subpartitions with a structure related to the Rogers--Ramanujan identities.  They conjectured the existence of an infinite class of congruences for $A_1(m)$, modulo powers of 5.  We give an explicit form of this conjecture, and prove it for all powers of 5.
\end{abstract}

\section{Introduction}

The unrestricted integer partition function $p(m)$ has been studied since the time of Euler \cite{Euler}, but its number theoretic properties are not at all obvious.  It is not immediately clear, for example, when $p(m)$ is prime, or a square or higher power, or even what the parity of $p(m)$ is.  One of the first important results of the arithmetic properties of $p(m)$ was discovered by Ramanujan \cite{Ramanujan}, and proved in 1938 by Watson \cite{Watson}:

\begin{theorem}
Let $24m\equiv 1\pmod{5^{n}}$.  Then $p(m)\equiv 0\pmod{5^{n}}$.
\end{theorem}  Ramanujan also found similar remarkable families of congruences for powers of 7 and 11.  To prove this theorem even in the single case $n=1$ is difficult.  The general theorem is of a considerably deep kind.

Since Ramanujan's discovery, similar families of congruences have been discovered in many different restricted partition functions.  While a similar approach to Watson's has proven useful for verifying these congruence families, various complications can arise in connection with the spaces of modular functions---and the associated modular curves---that underlie the different partition functions being studied.

We will concern ourselves here with an interesting new family of congruences for a weighted partition function, discovered by Choi, Kim, and Lovejoy \cite{CKL}, in 2015.  The partition function is related to the Rogers--Ramanujan identities, as well as to the unrestricted partition function $p(m)$.

Our method to resolve the complications associated with this new family of congruences is based largely on the methods developed by Paule and Radu to prove the Andrews--Sellers conjecture \cite{Paule}.  These methods themselves are related to the techniques developed by Atkin to prove Ramanujan's congruence family for $p(m)$ over powers of 11 \cite{Atkin}.

To begin, we define a class of subpartitions studied by Kolitsch \cite{Kolitsch}:

\begin{definition}
Let $\lambda$ be a partition of $m$.  The Rogers--Ramanujan subpartition of $\lambda$ is the unique subpartition with a maximal number of parts, in which the parts are nonrepeating, nonconsecutive, and larger than the remaining parts of $\lambda$.  More specifically, if $\lambda$ is the nonincreasing sequence $(a_1, a_2, ..., a_l, a_{l+1} ... a_k)$, then the Rogers--Ramanujan subpartition of $\lambda$ is the largest possible subpartition $(a_1, a_2, ..., a_l)$ with no repeated or consecutive parts, and with $a_l > a_{l+1}$ (If $l=k$, define $a_{l+1}=0$).  Here, $l$ is the length of the subpartition.
\end{definition}  At times we will denote such a subpartition a R--R subpartition.

For instance, the Rogers--Ramanujan subpartition of 

\begin{align*}
\lambda = (8,5,3,2,2,1,1,1)
\end{align*} is $(8,5,3)$, with length 3.  On the other hand, the Rogers--Ramanujan subpartition of 
\begin{align*}
\kappa = (8,8,2,2,1,1,1)
\end{align*} is simply the length-0 empty partition.

With this, we now define the partition functions $R_l(m)$, $A_1(m)$ as follows:

\begin{definition}
Let $R_l(m)$ be the number of partitions of $m$ containing a Rogers--Ramanujan subpartition of length $l$, and

\begin{align}
A_1(m) = \sum_{l\ge 0} l\cdot R_l(m).\label{defA1}
\end{align}
\end{definition}

For example, we consider $A_1(5)$.  Here we give the 7 partitions of 5, with the corresponding R--R subpartitions:

\begin{align*}
&(5)\supseteq (5),\\
&(4,1)\supseteq (4,1),\\
&(3,2)\supseteq (3),\\
&(3,1,1)\supseteq (3),\\
&(2,2,1),\\
&(2,1,1,1)\supseteq (2),\\
&(1,1,1,1,1).
\end{align*}  

So we find that of the seven partitions of 5, four of them contain a R--R subpartition of length 1, one partition contains a R--R subpartition of length 2, and two partitions contain no R--R subpartition.  We therefore have

\begin{align*}
A_1(5) = 1\cdot R_1(5)+2\cdot R_2(5) = 4+2 = 6.
\end{align*}

Choi, Kim, and Lovejoy proved \cite[Proposition 6.4]{CKL} that 

\begin{align}
A_1(25m+24)\equiv 0\pmod{5}.\label{case1}
\end{align}  From this and further numerical evidence, they conjectured the existence of a family of partition congruences, similar in kind to Ramanujan's classic congruences for $p(m)$.  We have proved the existence of this conjectured family of congruences, which we now give in a precise form.

\begin{theorem}
For $24m\equiv 1\pmod{5^{2n}},$ we have $A_1(m)\equiv 0\pmod{5^n}.$
\end{theorem}

We give the proof of Theorem 2 in Section 5.  In Section 2 we will define the generating function for $A_1(m)$, and reduce our problem to one more directly accessible in terms of eta quotients.  We will define the sequence $\mathcal{L} = (L_n)_{n\ge 1}$, together with the $U^{(j)}$ operators in Section 3.  In doing so, we will have provided an outline for the inductive proof of a stronger version of our theorem.

In Section~4 we will define the key spaces $X^{(j)}$ of modular functions over $\Gamma_0(20)$.  We will also use some key fundamental relations between specific modular functions, with a recursion defined by a modular equation for the prime 5, to prove some important lemmas which will lead up to our main proof.  Finally, in Section 6 we will list and discuss the proofs of our fundamental relations.

We wish to note that, as demonstrated in Section 2, our theorem provides a means of connecting $A_1(m)$, which has a clear association with the Rogers--Ramanujan identities, with Ramanujan's classic congruence family for $p(m)$ modulo powers of 5.  Moreover, our own numerical evidence indicates that there are no additional simple congruence families for $A_1(m)$ relative to powers of 7 or 11.  This would seem to highlight a peculiar significance of the number 5 in connecting $p(m)$ with the Rogers--Ramanujan identities.  Whether or not this observation is actually conducive to additional work is not certain, but the connection is curious nevertheless.

\section{Generating Function}

Hereafter, as in \cite[Chapter 2]{Andrews}, define 
\begin{align}
(a;b)_{r} = \prod_{n=0}^{r-1}(1-ab^n), \text{ and } (a;b)_{\infty} = \lim_{r\rightarrow\infty} (a;b)_r.\label{qpoch}
\end{align}

Recall \cite[Chapter 7]{Andrews} the generating function for the number of partitions into exactly $r$ parts, in which all parts are nonconsecutive and nonrepeating, is given by

\begin{align*}
\frac{q^{r^2}}{(q;q)_r}.
\end{align*}  However, if we allow the denominator to grow to $(q;q)_{\infty}$, we now generate the number of partitions of $m$ in which \textit{the first r parts are necessarily nonconsecutive and nonrepeating}, and larger than all remaining parts.

Notice, however, that such a partition may indeed have a larger number of large, nonconsecutive, nonrepeating parts; that is, all partitions containing a Rogers--Ramanujan subpartition of length $l\ge r$ are also accounted for with $q^{r^2}/(q;q)_{\infty}$.  In particular, if we sum from $r=1$ to $l$, i.e.,

\begin{align*}
\sum_{r=1}^{l}\frac{q^{r^2}}{(q;q)_{\infty}} = \frac{1}{(q;q)_{\infty}}\sum_{r=1}^{l}q^{r^2},
\end{align*} the number of partitions of $m$ containing a Rogers--Ramanujan subpartition of length $l$ is accounted for a total of $l$ times.  Since of course, $(q;q)_{\infty}^{-1}\sum_{r=l+1}^{\infty}q^{r^2}$ will only account for subpartitions of length $>l$, we have

\begin{theorem}
\begin{align}
\frac{1}{(q;q)_{\infty}}\sum_{r=1}^{\infty}q^{r^2} = \sum_{m=1}^{\infty}\sum_{l\ge 0} l\cdot R_l(m) q^m = \sum_{m=1}^{\infty} A_1(m)q^m.\label{gen1}
\end{align}
\end{theorem}

However, we can reduce the generating function to a simpler object.  We see that

\begin{align}
\sum_{m=1}^{\infty}A_1(m)q^m &= \frac{1}{(q;q)_{\infty}}\left(\sum_{r=1}^{\infty} q^{r^2}\right)\\
&=\frac{1}{2}\frac{1}{(q;q)_{\infty}}\left(\left( \sum_{r=-\infty}^{\infty} q^{r^2}\right) - 1 \right).\label{gen2}
\end{align}  We can multiply both sides of (\ref{gen2}) by 2:

\begin{align}
2\sum_{m=1}^{\infty}A_1(m)q^m &= \frac{1}{(q;q)_{\infty}}\left( \sum_{r=-\infty}^{\infty} q^{r^2}\right) - \frac{1}{(q;q)_{\infty}}.\label{gen3}
\end{align}  Now we bring Jacobi's triple product identity \cite[Chapter 3, Theorem 3]{Knopp} to bear on $\sum_{r=-\infty}^{\infty}q^{r^2}$, and have

\begin{align}
2\sum_{m=1}^{\infty}A_1(m)q^m &= \frac{(q^2;q^2)_{\infty}(-q;q^2)^2_{\infty}}{(q;q)_{\infty}} - \frac{1}{(q;q)_{\infty}}\\
&= \frac{(q^2;q^2)^5_{\infty}}{(q;q)^3_{\infty}(q^4;q^4)^2_{\infty}} - \frac{1}{(q;q)_{\infty}}.\label{gen4}
\end{align}  Finally, we can express $1/(q;q)_{\infty} = \sum_{m=0}^{\infty}p(m)q^m$.  If we let $a(m)$ represent the coefficient of $q^m$ in our first term of (\ref{gen4}), then we have

\begin{align}
2\cdot A_1(m)  &= a(m) - p(m).\label{defa}
\end{align}

Now, we already know from Ramanujan's classic congruences that whenever $24m\equiv 1\pmod{5^{2n}}$, we have $p(m)\equiv 0\pmod{5^{2n}}$.  Moreover, $\gcd(2,5)=1$, so that the factor of 2 on the left hand side of (\ref{defa}) is irrelevant to the question of divisibility by powers of 5.  Therefore, we need only verify our congruence family for $a(m)$.

\section{Setup for a Proof by Induction}

\subsection{$L_n$}

Because the initial cases of Theorem 2 may be checked by computation, it is reasonable to attempt a proof by induction.  We need a way of connecting the case for some arbitrary $n$ to $n+1$.  To do this, we define the functions $L_n$ ($n\ge 0$) as follows:

\begin{align}
L_0 &= 1,\\
L_{2n-1} = L_{2n-1}(q) &= \frac{(q^5;q^5)^3_{\infty}(q^{20};q^{20})^2_{\infty}}{(q^{10};q^{10})^5_{\infty}}\sum_{m=0}^{\infty} a(5^{2n-1}m+\lambda_{2n-1})q^{m+1},\label{L1}\\
L_{2n} = L_{2n} (q) &= \frac{(q;q)^3_{\infty}(q^{4};q^{4})^2_{\infty}}{(q^{2};q^{2})^5_{\infty}}\sum_{m=0}^{\infty} a(5^{2n}m+\lambda_{2n})q^{m+1},\label{L2}
\end{align} with

\begin{align}
\lambda_{2n-1} = \frac{19\cdot 5^{2n-1}+1}{24}, \text{ and } \lambda_{2n} = \frac{23\cdot 5^{2n}+1}{24}.\label{deflambda}
\end{align}  Here, $\lambda_n$ is the minimal solution to $24x\equiv 1\pmod{5^{n}}$.  Because $\gcd(24,5)=1$, any other solution $\delta$ must satisfy 

\begin{align*}
\delta\equiv \lambda_n\pmod{5^{n}},
\end{align*} i.e., $\delta = 5^nm+\lambda_n.$

For $n\ge 1$, we can give $L_n$ a more succinct representation as

\begin{align}
L_n = \mathcal{A}_{n}(q)\cdot\sum_{24m\equiv 1\ (\text{mod}5^n)} a(m)q^{\left\lfloor \frac{m}{5^n} \right\rfloor + 1},
\end{align} with $\mathcal{A}_{n}(q)$ defined by the prefactors of (\ref{L1}) or (\ref{L2}), depending on the parity of~$n$.

Notice that the prefactors $\mathcal{A}_n(q)$ can be expanded into integer power series with constant term 1.  This implies that no positive power of 5 can divide any $\mathcal{A}_n(q)$ (that is, no positive power of 5 can divide every term of $\mathcal{A}_n(q)$).  Therefore, if a given power of 5 divides $L_n$, then that given power of 5 must divide every term $a(m)$.  Demonstrating that $L_{2n}\equiv 0\pmod{5^n}$ would imply Theorem 2.

We now need a means of connecting $L_n$ to $L_{n+1}$.

\subsection{The $U_5$-Operator}

Hereafter, we define $\mathbb{H}$ as the upper half complex plane, and let $q=e^{2\pi i\tau}$, with $\tau\in\mathbb{H}$.

We recall the classic $U_5$-operator:

\begin{definition}

Let $f(q) = \sum_{m\ge M}a(m)q^m$.  Then define

\begin{align}
U_5\{f(q)\} = \sum_{5m\ge M} a(5m)q^m.\label{defU5}
\end{align}

\end{definition}

We list some of the important properties of $U_5$.  These properties are standard to the theory of partition congruences, and proofs can be found in \cite[Chapter 10]{Andrews} and \cite[Chapter 8]{Knopp}.

\begin{lemma}

Given two functions 

\begin{align*}
f(q) = \sum_{m\ge M}a(m)q^m,\ g(q) = \sum_{m\ge N}b(m)q^m,
\end{align*} any $\alpha\in\mathbb{C}$, a primitive fifth root of unity $\zeta$, and the convention that $q^{1/5}~=~e^{2\pi i\tau/5}$, we have the following:
\begin{enumerate}
\item $U_5\{\alpha\cdot f+g\} = \alpha\cdot U_5\{f\} + U_5\{g\}$;
\item $U_5\{f(q^5)g(q)\} = f(q) U_5\{g(q)\}$;
\item $5\cdot U_5\{f\} = \sum_{r=0}^4 f\left( \zeta^rq^{1/5} \right)$.
\end{enumerate}
\end{lemma}

Returning to our specific problem, we make use of the properties of $U_5$ by defining the following:

\begin{definition}
\begin{align}
A(q) &= q\frac{(q^{2};q^{2})^5_{\infty}}{(q;q)^3_{\infty}(q^{4};q^{4})^2_{\infty}}\cdot \frac{(q^{25};q^{25})^3_{\infty}(q^{100};q^{100})^2_{\infty}}{(q^{50};q^{50})^5_{\infty}};\label{deffA}\\
U^{(0)} \{f\} &= U_5\left\{ A(q) \cdot f \right\}, \text{ and } U^{(1)}\{f\} = U_5\{f\}.\label{defU01}
\end{align}
\end{definition}  Notice that for two functions $f$ and $g$, and any $\alpha\in\mathbb{C}$,

\begin{align}
&U^{(0)} \{\alpha\cdot f + g\}\notag\\ &= U_5\left\{ A(q) (\alpha\cdot f + g) \right\} = U_5\left\{ \alpha\cdot A(q) \cdot f +A(q)\cdot g \right\}\\
&= \alpha\cdot U_5\left\{A(q) \cdot f\right\} + U_5\left\{ A(q)\cdot g \right\} = \alpha\cdot U^{(0)} \{f\} + U^{(0)}\{g\}.\label{linU0}
\end{align}  Since we already know from Part 1 of Lemma 1 that $U^{(1)} = U_5$ is linear, we have thus established that $U^{(j)}$ is linear for $j=0,1$.

This now gives us a means of connecting $L_n$ with $L_{n+1}$.

\begin{theorem}
For all $n\in\mathbb{Z}_{>0}$,

\begin{align}
L_{2n-1} &= U^{(0)} \{L_{2n-2}\},\text{ and } L_{2n} = U^{(1)} \{L_{2n-1}\}.
\end{align}

\end{theorem}

\begin{proof}

For a given $n\in\mathbb{Z}_{>0}$, 

\begin{align}
&U^{(1)}\left\{ L_{2n-1} \right\}\notag\\ &= U_5\left\{ \frac{(q^5;q^5)^3_{\infty}(q^{20};q^{20})^2_{\infty}}{(q^{10};q^{10})^5_{\infty}}\sum_{m=0}^{\infty} a(5^{2n-1}m+\lambda_{2n-1})q^{m+1} \right\}\\
&= \frac{(q;q)^3_{\infty}(q^{4};q^{4})^2_{\infty}}{(q^{2};q^{2})^5_{\infty}}U_5\left\{\sum_{m\ge 0} a(5^{2n-1}m+\lambda_{2n-1})q^{m+1} \right\}\\
&= \frac{(q;q)^3_{\infty}(q^{4};q^{4})^2_{\infty}}{(q^{2};q^{2})^5_{\infty}}U_5\left\{\sum_{m\ge 1} a(5^{2n-1}(m-1)+\lambda_{2n-1})q^{m} \right\}\\
&= \frac{(q;q)^3_{\infty}(q^{4};q^{4})^2_{\infty}}{(q^{2};q^{2})^5_{\infty}}\sum_{5m\ge 1} a(5^{2n-1}(5m-1)+\lambda_{2n-1})q^{m}\\
&= \frac{(q;q)^3_{\infty}(q^{4};q^{4})^2_{\infty}}{(q^{2};q^{2})^5_{\infty}}\sum_{m= 0}^{\infty} a(5^{2n}m+5^{2n}-5^{2n-1}+\lambda_{2n-1})q^{m+1}\\
&= \frac{(q;q)^3_{\infty}(q^{4};q^{4})^2_{\infty}}{(q^{2};q^{2})^5_{\infty}}\sum_{m= 0}^{\infty} a(5^{2n}m+\lambda_{2n})q^{m+1},
\end{align} since

\begin{align}
5^{2n}-5^{2n-1}+\lambda_{2n-1} &= 5^{2n-1}(4) + \frac{19\cdot 5^{2n-1}+1}{24}\\
&= \frac{5^{2n-1}(5\cdot 23)+1}{24} = \frac{23\cdot 5^{2n}+1}{24} = \lambda_{2n}.
\end{align}  Furthermore,

\begin{align}
&U^{(0)}\left\{ L_{2n} \right\}\notag\\ &= U_5\left\{A(q)\cdot \frac{(q;q)^3_{\infty}(q^{4};q^{4})^2_{\infty}}{(q^{2};q^{2})^5_{\infty}}\sum_{m= 0}^{\infty} a(5^{2n}m+\lambda_{2n})q^{m+1} \right\}\\
&=U_5\left\{\frac{(q^{25};q^{25})^3_{\infty}(q^{100};q^{100})^2_{\infty}}{(q^{50};q^{50})^5_{\infty}}\sum_{m= 0}^{\infty} a(5^{2n}m+\lambda_{2n})q^{m+2} \right\}\\
&=\frac{(q^{5};q^{5})^3_{\infty}(q^{20};q^{20})^2_{\infty}}{(q^{10};q^{10})^5_{\infty}}U_5\left\{\sum_{m= 0}^{\infty} a(5^{2n}m+\lambda_{2n})q^{m+2} \right\}\\
&=\frac{(q^{5};q^{5})^3_{\infty}(q^{20};q^{20})^2_{\infty}}{(q^{10};q^{10})^5_{\infty}}U_5\left\{\sum_{m\ge 2} a(5^{2n}(m-2)+\lambda_{2n})q^{m} \right\}\\
&=\frac{(q^{5};q^{5})^3_{\infty}(q^{20};q^{20})^2_{\infty}}{(q^{10};q^{10})^5_{\infty}}\sum_{5m\ge 2} a(5^{2n}(5m-2)+\lambda_{2n})q^{m}.
\end{align}  Notice that $5m\ge 2$ implies that $m\ge 1$ for $m\in\mathbb{Z}$, so that

\begin{align}
&U^{(0)}\left\{ L_{2n} \right\}\notag\\
&=\frac{(q^{5};q^{5})^3_{\infty}(q^{20};q^{20})^2_{\infty}}{(q^{10};q^{10})^5_{\infty}}\sum_{m\ge 1} a(5^{2n+1}m-2\cdot 5^{2n}+\lambda_{2n})q^{m}\\
&=\frac{(q^{5};q^{5})^3_{\infty}(q^{20};q^{20})^2_{\infty}}{(q^{10};q^{10})^5_{\infty}}\sum_{m=0}^{\infty} a(5^{2n+1}(m+1)-2\cdot 5^{2n}+\lambda_{2n})q^{m+1}\\
&=\frac{(q^{5};q^{5})^3_{\infty}(q^{20};q^{20})^2_{\infty}}{(q^{10};q^{10})^5_{\infty}}\sum_{m=0}^{\infty} a(5^{2n+1}m+\lambda_{2n+1})q^{m+1},
\end{align} since

\begin{align}
5^{2n+1}-2\cdot 5^{2n}+\lambda_{2n} &= 5^{2n}(3) + \frac{23\cdot 5^{2n}+1}{24}\\
&= \frac{5^{2n}(5\cdot 19)+1}{24} = \frac{19\cdot 5^{2n+1}+1}{24} = \lambda_{2n+1}.
\end{align}

\end{proof}

We remark that the definition of $L_n$ and the means of going from $L_n$ to $L_{n+1}$ is a standard technique in this subject area \cite[Chapter 8]{Knopp}.

We can now study the sequence $\mathcal{L}=(L_n)_{n\ge 0}$, with $U^{(j)}$ as a means of connecting one element with the next.  We know that $L_2$ is divisible by 5 (since this is equivalent for the first case of Theorem 2, which Choi, Kim, and Lovejoy have already proven).  We want to prove that as $n$ increases by 2, $L_n$ will become divisible by an additional power of 5.
In a more formal language, we will prove the following theorem, which implies Theorem 2.

\begin{theorem}

The sequence $\mathcal{L}=(L_n)_{n\ge 0}$ converges to 0 in the 5-adic sense: for any $M\in\mathbb{N}$ there exists an $N\in\mathbb{N}$ such that for all $n\ge N$,

\begin{align}
L_n\equiv 0\pmod{5^{M}}.\label{padic}
\end{align}  In particular, $N = \lfloor M/2 \rfloor$ will suffice.

\end{theorem}

\section{Subspace Structure}

\subsection{$X^{(j)}$}
We will now construct the spaces of modular functions that are necessary for our purposes.  In the case of $p(m)$, the necessary modular functions are defined over $\Gamma_0(5)$.  The associated modular curve $X_0(5)$ is simple enough that the complete space of modular functions over $\Gamma_0(5)$ may be used.

However, in our case we will work over $\Gamma_0(20)$.  The modular curve $X_0(20)$ is more complex, ensuring that 5-adic convergence to 0 is not a necessity for the entire space of modular functions defined over it (See \cite[Section 1.3]{Paule} for a discussion on the effects of the properties of $X_0(N)$.  For a comprehensive discussion of the theory of Riemann surfaces to the subject of modular forms, see \cite[Chapters 1-3]{Diamond}).  We therefore need to restrict ourselves to subspaces of modular functions that will indeed converge 5-adically to 0.

Let $q=e^{2\pi i\tau}$, with $\tau\in\mathbb{H}$, and define

\begin{align}
t &= \frac{\eta(5\tau)^6}{\eta(\tau)^6} = q\frac{(q^5;q^5)_{\infty}^6}{(q;q)_{\infty}^6},\label{t}\\
\rho &= \frac{\eta(\tau)^2\eta(4\tau)^2\eta(10\tau)^{8}}{\eta(5\tau)^2\eta(20\tau)^{10}} = \frac{1}{q^5}\frac{(q;q)^2_{\infty}(q^4;q^4)^2_{\infty}(q^{10};q^{10})^{8}_{\infty}}{(q^5;q^5)^2_{\infty}(q^{20};q^{20})^{10}_{\infty}}\label{rho}\\
\sigma &= \frac{\eta(4\tau)^4\eta(10\tau)^2}{\eta(2\tau)^2\eta(20\tau)^4} = \frac{1}{q^2}\frac{(q^4;q^4)^4_{\infty}(q^{10};q^{10})^{2}_{\infty}}{(q^2;q^2)^2_{\infty}(q^{20};q^{20})^{4}_{\infty}},\label{sigma}\\
\mu &= \frac{\eta(4\tau)\eta(5\tau)^5}{\eta(\tau)\eta(20\tau)^5}= \frac{1}{q^3}\frac{(q^4;q^4)_{\infty}(q^{5};q^{5})^{5}_{\infty}}{(q;q)_{\infty}(q^{20};q^{20})^{5}_{\infty}},\label{mu}
\end{align} with \cite[Chapter 3]{Knopp}

\begin{align}
\eta(\tau) = e^{\pi i\tau/12}\prod_{n=1}^{\infty}\left( 1 - e^{2\pi i n\tau} \right),\label{eta}
\end{align} and

\begin{align}
p_0 =& 31 \rho^{-1} - 22 \sigma \rho^{-1} - 9 \sigma^2 \rho^{-1} - 208 \rho^{-2} - 96 \sigma \rho^{-2} + 304 \sigma^2 \rho^{-2}\notag\\& - 32 \rho^{-1} \mu + 416 \rho^{-2} \mu + 416 \sigma \rho^{-2} \mu - 208 \rho^{-2} \mu^2,\label{p0}\\
p_1 =& 261 \rho^{-1} + 126 \sigma \rho^{-1} + 13 \sigma^2 \rho^{-1} - 960 \rho^{-2} - 5120 \sigma \rho^{-2} - 320 \sigma^2 \rho^{-2}\notag\\& + 
 64 \rho^{-1} \mu + 320 \rho^{-2} \mu - 1280 \sigma \rho^{-2} \mu + 640 \rho^{-2} \mu^2.\label{p1}
\end{align}  In Section 6 we will demonstrate that each of these functions is a modular function over $\Gamma_0(20)$.

Define

\begin{align}
S_0 &= \left< 1, p_0 \right>_{\mathbb{Z}[t]},\label{S0}\\
S_1 &= \left< 1, p_1 \right>_{\mathbb{Z}[t]}.\label{S1}
\end{align}  That is, for $j=0,1$, $S_j$ is the free $\mathbb{Z}[t]$-module generated by $1$ and $p_j$.  A given $g\in S_j$ has the form

\begin{align*}
g = g_{\alpha}(t) + p_j\cdot g_{\beta}(t),
\end{align*} with $g_{\alpha}, g_{\beta}\in\mathbb{Z}[x]$.

We find (see Section 6.2) that

\begin{align*}
L_1 = U^{(0)}\{1\} = p_1\in S_1.
\end{align*}  As the relations in Groups II and IV of Section 6 demonstrate,

\begin{align*}
L_2 = U^{(1)}\{ L_1\} = U^{(1)}\{ p_1\}\in S_0, \text{ and } U^{(0)}\{ p_0\}\in S_1.
\end{align*}  This, with the linearity of $U^{(j)}$, ensures that that for $n\ge 0$,

\begin{align}
L_{2n}&\in S_0,\\
L_{2n-1}&\in S_1.
\end{align}

We will work with specific subspaces of $S_0, S_1$.

\begin{definition}
A function $f:\mathbb{Z}\rightarrow\mathbb{Z}$ is discrete if it is nonzero for only finitely many integers.  A two-variable function $h:\mathbb{Z}\times\mathbb{Z}\rightarrow\mathbb{Z}$ is a discrete array if for any fixed $m_0\in\mathbb{Z}$, the function $h(m_0,n)$ is discrete over $n\in\mathbb{Z}$.
\end{definition}  We now define our relevant subspaces:

\begin{align}
&X^{(0)}\notag\\ =& \Bigg\{ \sum_{n=0}^{\infty}r(n) 5^{\left\lfloor \frac{5n}{2} \right\rfloor} p_0 t^n + \sum_{n=1}^{\infty}s(n) 5^{\left\lfloor \frac{5n-3}{2} \right\rfloor} t^n: \textit{$r,s$ discrete functions} \Bigg\},
\end{align}
\begin{align}
&X^{(1)}\notag\\ =& \Bigg\{ \sum_{n=0}^{\infty}r(n) 5^{\left\lfloor \frac{5n}{2} \right\rfloor} p_1 t^n + \sum_{n=1}^{\infty}s(n) 5^{\left\lfloor \frac{5n-1}{2} \right\rfloor} t^n : \textit{$r,s$ discrete functions} \Bigg\}.
\end{align}  Notice that for $j=0,1$, we have $X^{(j)}\subseteq S_j$.  In particular, $L_1 = p_1\in X^{(1)}$.

\subsection{Modular Equation}

We have very carefully chosen the spaces $X^{(j)}$.  Rather than working directly with $L_n$, we will show that $L_n\in X^{(r)}$, with $r$ the residue of $n\pmod{2}$.  We then study how $U^{(j)}$ changes the structure of an arbitrary $f\in X^{(j)}$.

To do this, we will need to know the effects of $U^{(j)}$ on $p_jt^n, t^n$.  Our choice of $t = \eta(5\tau)^6/\eta(\tau)^6$ is especially convenient, as we have a powerful modular equation that can be brought to bear on the problem.

\begin{theorem}
Let
\begin{align*}
a_0(\tau) &= -t, \\
a_1(\tau) &=-5^3t^2-6\cdot 5t,\\
a_2(\tau) &=-5^6t^3-6\cdot 5^4t^2-63\cdot 5t,\\
a_3(\tau) &=-5^9t^4-6\cdot 5^7t^3-63\cdot 5^4t^2-52\cdot 5^2t,\\
a_4(\tau) &=-5^{12}t^5-6\cdot 5^{10}t^4-63\cdot 5^7t^3-52\cdot 5^5t^2-63\cdot 5^2t.
\end{align*}  Then

\begin{align}
t(\tau)^5+\sum_{j=0}^4 a_j(5\tau) t(\tau)^j = 0.\label{modeq}
\end{align}

\end{theorem}

A proof of this can be found in \cite[Section 3]{Paule}.

The value of this equation becomes immediate when we consider the following theorem:

\begin{lemma}
For any function $g:\mathbb{H}\rightarrow \mathbb{C}$,

\begin{align}
U_5\{g\cdot t^n\} = -\sum_{j=0}^4 a_j(\tau) U_5\{g\cdot t^{n+j-5}\}.
\end{align}

\end{lemma}

\begin{proof}
With equation (\ref{modeq}), we have

\begin{align}
&g(\tau)\cdot t(\tau)^n = -\sum_{j=0}^4 a_j(5\tau)\cdot g(\tau)\cdot t(\tau)^{n+j-5}.
\end{align}  Taking the $U_5$ operator, and remembering that 
\begin{align*}
U_5\{a_j(5\tau)\cdot g(\tau)\cdot t(\tau)^{n+j-5}\} = a_j(\tau)\cdot U_5\{g(\tau)\cdot t(\tau)^{n+j-5}\},
\end{align*} by Part 3 of Lemma 1, we find that

\begin{align}
U_5\{g\cdot t(\tau)^n\} &= -\sum_{j=0}^4 U_5\{a_j(5\tau)\cdot g\cdot t(\tau)^{n+j-5}\},\\
&=-\sum_{j=0}^4 a_j(\tau)\cdot U_5\{g\cdot t(\tau)^{n+j-5}\}.
\end{align}

\end{proof}

\subsection{Lemmas}

We now state and prove a key application of the modular equation for $t$.  This lemma is given in the form of two lemmas in \cite[Section 4]{Paule}, but we give the proof for the sake of completion.

\begin{lemma}
For any functions $g,y_0,y_1:\mathbb{H}\rightarrow\mathbb{C}$, if there exist $u_0, u_1, v_0, v_1\in\mathbb{Z}$, and discrete arrays $h_0(m,n), h_1(m,n)$ such that

\begin{align}
U_5\{g t^n\} =& \sum_{m\ge \left\lceil \frac{n+u_0}{5} \right\rceil}h_0(m,n)5^{\left\lfloor\frac{5m-n+v_0}{2}\right\rfloor}y_0t^m\notag\\ &+ \sum_{m\ge \left\lceil \frac{n+u_1}{5} \right\rceil}h_1(m,n)5^{\left\lfloor\frac{5m-n+v_1}{2}\right\rfloor}y_1t^m
\end{align} for five consecutive integers, then such a relation holds for every larger integer.
\end{lemma}

\begin{proof}
Suppose that for specific functions $g,\ y_0,\ y_1,$ discrete arrays $h_0,\ h_1$, and integers $u_0,\ u_1,\ v_0,\ v_1$, the given relation holds for five consecutive integers:

\begin{align*}
n_0,\ n_0+1,\ n_0+2,\ n_0+3,\ n_0+4.
\end{align*}

We prove the lemma by induction.

Let $k\ge n_0+5$, and assume that the relation holds for all $j\in\mathbb{Z}$ such that $n_0\le j\le k-1$.  In particular, the relation holds for $j=k-5,\ k-4,\ ...,\ k-1$.  We want to prove that the relation must hold for $k$.  It can be quickly verified from the previous lemma that

\begin{align}
a_j(\tau) = \sum_{l=1}^5 s(j,l)5^{\left\lfloor \frac{5l+j-4}{2} \right\rfloor}t^l,
\end{align} for some unique function $s:\{0,...,4\}\times \{1,...,5\}\rightarrow\mathbb{Z}$.  With this in mind, we have

\begin{align}
&U_5\{g t^k\}\notag\\
&= -\sum_{j=0}^4 a_j(\tau) U_5\{g\cdot t(\tau)^{k+j-5}\}\\
&=-\sum_{j=0}^4 a_j(\tau)\sum_{i=0,1}\sum_{m\ge \left\lceil \frac{k+j-5+u_i}{5} \right\rceil}h_i(m,k+j-5)5^{\left\lfloor\frac{5m-(k+j-5)+v_i}{2}\right\rfloor}y_it^m\\
&=-\sum_{i=0,1}\sum_{j=0}^4 a_j(\tau)\sum_{m\ge \left\lceil \frac{k+u_i}{5}-\frac{5-j}{5} \right\rceil}h_i(m,k+j-5)5^{\left\lfloor\frac{5m-(k+j-5)+v_i}{2}\right\rfloor}y_it^m.
\end{align}  Taking $m_{i,j} = \left\lceil \frac{k+u_i}{5}-\frac{5-j}{5} \right\rceil$, we have

\begin{align}
&U_5\{g t^k\}\notag\\
&=-\sum_{\substack{i=0,1,\\ 0\le j\le 4,\\ 1\le l\le 5}}s(j,l)5^{\left\lfloor \frac{5l+j-4}{2} \right\rfloor}t^{l}\sum_{m\ge m_{i,j}}h_i(m,k+j-5)5^{\left\lfloor\frac{5m-(k+j-5)+v_i}{2}\right\rfloor}y_it^m\\
&=-\sum_{\substack{i=0,1,\\ 0\le j\le 4,\\ 1\le l\le 5}}\sum_{m\ge m_{i,j}}s(j,l)h_i(m,k+j-5)5^{\left\lfloor\frac{5m-(k+j-5)+v_i}{2}\right\rfloor + \left\lfloor \frac{5l+j-4}{2} \right\rfloor}y_it^{m+l}.
\end{align}  Now, we note that for any $M_1, M_2\in\mathbb{Z}$, we have $\left\lfloor \frac{M_1}{2} \right\rfloor + \left\lfloor \frac{M_2}{2} \right\rfloor \ge \left\lfloor \frac{M_1+M_2}{2} - \frac{1}{2} \right\rfloor$.  Therefore,

\begin{align}
&\left\lfloor\frac{5m-(k+j-5)+v_i}{2}\right\rfloor + \left\lfloor \frac{5l+j-4}{2} \right\rfloor\notag\\ \ge&\left\lfloor\frac{5m-(k+j-5)+v_i}{2} + \frac{5l+j-4}{2} - \frac{1}{2} \right\rfloor = \left\lfloor\frac{5(m+l)-k+v_i}{2} \right\rfloor.
\end{align}  Now since $m_{i,j} = \left\lceil \frac{k+u_i}{5}-\frac{5-j}{5} \right\rceil \ge \left\lceil \frac{k+u_i}{5} \right\rceil - 1$, and since $l\ge 1$, we relabel our powers of $t$ so that

\begin{align}
&U_5\{g t^k\}\notag\\ &= -\sum_{\substack{i=0,1,\\ 0\le j\le 4,\\ 1\le l\le 5}}\sum_{m\ge \left\lceil \frac{k+u_i}{5} \right\rceil - 1 + l}s(j,l)h_i(m-l,k+j-5)5^{\left\lfloor\frac{5m-k+v_i}{2}\right\rfloor}y_it^{m}.
\end{align} Finally, defining the discrete function $H_i(m,k)$ by

\begin{align*}
H_i(m,k) &= \left\{
  \begin{array}{lr}
    -\sum_{j=0}^4\sum_{l=1}^5 s(j,l)h_i(m-l,k+j-5), & m\ge l,\\
    0, & \text{otherwise,}
  \end{array}
\right.
\end{align*} we have

\begin{align}
U_5\{g t^k\} =& \sum_{m\ge \left\lceil \frac{k+u_0}{5} \right\rceil}H_0(m,k)5^{\left\lfloor\frac{5m-k+v_0}{2}\right\rfloor}y_0t^{m}\notag\\ &+ \sum_{m\ge \left\lceil \frac{k+u_1}{5} \right\rceil}H_1(m,k)5^{\left\lfloor\frac{5m-k+v_1}{2}\right\rfloor}y_1t^{m}.
\end{align}  By induction, we have established the given relation for all $n\ge n_0$.

\end{proof} 

\begin{samepage}
We can use this lemma to define a very useful ``skeletal" structure for $U^{(j)}\{p_jt^n\}, U^{(j)}\{t^n\}$ as follows:
\end{samepage}

\begin{samepage}
\begin{lemma}
There exist discrete arrays $a_{j}(m,n), b_{j}(m,n), c(m,n), d_{j}(m,n)$, with $j\in\{0,1\}$, such that for all nonnegative $n\in\mathbb{Z}$,
\begin{align}
U^{(0)}\{t^n\} =& \sum_{m\ge \left\lceil \frac{n+1}{5} \right\rceil} a_0(m,n) 5^{\left\lfloor \frac{5m-n-1}{2} \right\rfloor} t^m\notag\\ &+ \sum_{m\ge \left\lceil \frac{n}{5} \right\rceil} a_1(m,n) 5^{\left\lfloor \frac{5m-n}{2} \right\rfloor} p_1 t^m,\label{ineq1}\\
U^{(0)}\{p_0t^n\} =& \sum_{m\ge \left\lceil \frac{n+2}{5} \right\rceil} b_0(m,n) 5^{\left\lfloor \frac{5m-n-1}{2} \right\rfloor} t^m\notag\\ &+ \sum_{m\ge \left\lceil \frac{n}{5} \right\rceil} b_1(m,n) 5^{\left\lfloor \frac{5m-n}{2} \right\rfloor} p_1 t^m,\label{ineq2}\\
U^{(1)}\{t^n\} =& \sum_{m\ge \left\lceil \frac{n}{5} \right\rceil} c(m,n) 5^{\left\lfloor \frac{5m-n-1}{2} \right\rfloor} t^m,\label{ineq3}\\
U^{(1)}\{p_1t^n\} =& \sum_{m\ge \left\lceil \frac{n+1}{5} \right\rceil} d_0(m,n) 5^{\left\lfloor \frac{5m-n-1}{2} \right\rfloor} t^m\notag\\ &+ \sum_{m\ge \left\lceil \frac{n-1}{5} \right\rceil} d_1(m,n) 5^{\left\lfloor \frac{5m-n+2}{2} \right\rfloor} p_0 t^m,\label{ineq4}
\end{align} 
\end{lemma}
\end{samepage}

Notice that we can set $a_0(m,n)=0$ whenever $m< \left\lceil (n+1)/5 \right\rceil$.  More generally, for $j=0,1$, we can define

\begin{align}
a_j(m,n) = b_j(m,n) = c(m,n) = d_j(m,n) = 0\label{fullineq}
\end{align} if the corresponding inequalities for $m,n$ in (\ref{ineq1}), (\ref{ineq2}), (\ref{ineq3}), (\ref{ineq4}) do not hold.

\begin{proof}
The previous lemma establishes that if these relations hold for $k-5, k-4, ... k-1$, then they will hold for all $n\ge k$.  We therefore need twenty initial relations---relations for five consecutive values, in four categories.  These relations are demonstrated to hold in Section 6, for $-4\le k\le 0$.
\end{proof}

\section{Main Theorem}

\begin{theorem}
If $f\in X^{(0)}$, then we have $U^{(0)}\{f\}\in X^{(1)}$.  If $f\in X^{(1)}$, then we have $5^{-1}U^{(1)}\{f\}\in X^{(0)}$.
\end{theorem}

\begin{proof}
Let $f\in X^{(0)}$.  Then there exist discrete functions $r,s$ such that

\begin{align}
f = \sum_{n=0}^{\infty}r(n) 5^{\left\lfloor \frac{5n}{2} \right\rfloor} p_0 t^n + \sum_{n=1}^{\infty}s(n) 5^{\left\lfloor \frac{5n-3}{2} \right\rfloor} t^n.
\end{align}  We take $U^{(0)}\{f\}$.  Using Lemma 4, with condition (\ref{fullineq}), we find that

\begin{align}
&U^{(0)}\{f\}\notag\\ &= \sum_{n=0}^{\infty}r(n) 5^{\left\lfloor \frac{5n}{2} \right\rfloor} U^{(0)}\{p_0 t^n\} + \sum_{n=1}^{\infty}s(n) 5^{\left\lfloor \frac{5n-3}{2} \right\rfloor} U^{(0)}\{t^n\}\\
&=\sum_{n=0}^{\infty}r(n) 5^{\left\lfloor \frac{5n}{2} \right\rfloor} \Bigg( \sum_{m\ge \left\lceil \frac{n+2}{5} \right\rceil} b_0(m,n) 5^{\left\lfloor \frac{5m-n-1}{2} \right\rfloor} t^m\notag\\ &\ \ \ \ \ \ \ \ \ \ \ \ \ \ \ \ \ \ \ \ \ \ \ \ \ \ \ \ \ \ \ \ \ \ \ \ + \sum_{m\ge \left\lceil \frac{n}{5} \right\rceil} b_1(m,n) 5^{\left\lfloor \frac{5m-n}{2} \right\rfloor} p_1 t^m \Bigg)\notag\\
&+\sum_{n=1}^{\infty}s(n) 5^{\left\lfloor \frac{5n-3}{2} \right\rfloor} \Bigg( \sum_{m\ge \left\lceil \frac{n+1}{5} \right\rceil} a_0(m,n) 5^{\left\lfloor \frac{5m-n-1}{2} \right\rfloor} t^m\notag\\ &\ \ \ \ \ \ \ \ \ \ \ \ \ \ \ \ \ \ \ \ \ \ \ \ \ \ \ \ \ \ \ \ \ \ \ \ + \sum_{m\ge \left\lceil \frac{n}{5} \right\rceil} a_1(m,n) 5^{\left\lfloor \frac{5m-n}{2} \right\rfloor} p_1 t^m \Bigg).
\end{align}  Because $a_j(m,n), b_j(m,n), c(m,n), d_j(m,n)$ have the additional condition (\ref{fullineq}), we may rearrange our summands such that

\begin{align}
U^{(0)}\{f\} &=p_1\sum_{m\ge 0}\sum_{n\ge 0}r(n) b_1(m,n)  5^{\left\lfloor \frac{5n}{2} \right\rfloor + \left\lfloor \frac{5m-n}{2} \right\rfloor} t^m\label{u0a}\\
&+ p_1\sum_{m\ge 1}\sum_{n\ge 1}s(n) a_1(m,n)  5^{\left\lfloor \frac{5n-3}{2} \right\rfloor + \left\lfloor \frac{5m-n}{2} \right\rfloor} t^m\label{u0b}\\
&+\sum_{m\ge 1}\sum_{n\ge 0}r(n) b_0(m,n)  5^{\left\lfloor \frac{5n}{2} \right\rfloor + \left\lfloor \frac{5m-n-1}{2} \right\rfloor} t^m\label{u0c}\\
&+\sum_{m\ge 1}\sum_{n\ge 1}s(n) a_0(m,n)  5^{\left\lfloor \frac{5n-3}{2} \right\rfloor + \left\lfloor \frac{5m-n-1}{2} \right\rfloor} t^m\label{u0d}.
\end{align}  Now, we simplify the powers of 5 corresponding to each double sum.  For line (\ref{u0a}), with $m,n\ge 0$,

\begin{align}
\left\lfloor \frac{5n}{2} \right\rfloor + \left\lfloor \frac{5m-n}{2} \right\rfloor &= \left\lfloor \frac{3n}{2} \right\rfloor + \left\lfloor \frac{5m+n}{2} \right\rfloor \ge \left\lfloor \frac{5m}{2} \right\rfloor.
\end{align}  For (\ref{u0b}), notice that $m,n \ge 1$.  So we have
\begin{align}
\left\lfloor \frac{5n-3}{2} \right\rfloor + \left\lfloor \frac{5m-n}{2} \right\rfloor &= \left\lfloor \frac{3n-3}{2} \right\rfloor + \left\lfloor \frac{5m+n}{2} \right\rfloor \ge \left\lfloor \frac{5m}{2} \right\rfloor.
\end{align}  Notice that $\left\lfloor \frac{5m}{2} \right\rfloor$ is the necessary power of $5$ in the coefficient of $p_1t^m$ for $X^{(1)}$.

For (\ref{u0c}), we have $m\ge 1$, $n\ge 0$.

\begin{align}
\left\lfloor \frac{5n}{2} \right\rfloor + \left\lfloor \frac{5m-n-1}{2} \right\rfloor &\ge \left\lfloor \frac{5m+n-1}{2} \right\rfloor \ge \left\lfloor \frac{5m-1}{2} \right\rfloor.
\end{align}  Finally, for (\ref{u0d}), with $m,n\ge 1$,

\begin{align}
\left\lfloor \frac{5n-3}{2} \right\rfloor + \left\lfloor \frac{5m-n-1}{2} \right\rfloor &\ge \left\lfloor \frac{5m+n-1}{2} \right\rfloor \ge \left\lfloor \frac{5m-1}{2} \right\rfloor.
\end{align}  Since $\left\lfloor \frac{5m-1}{2} \right\rfloor$ is the necessary power of $5$ in the coefficient of $t^m$ for $X^{(1)}$ (and no constant term is generated), we have $U^{(0)}\{f\}\in X^{(1)}$.

To prove the second statement of our theorem, we let $f\in X^{(1)}$.  We want $U^{(1)}\{f\}\in X^{(0)}$, \textit{with an additional power of 5 in each term}.  To begin, we have by hypothesis,

\begin{align}
f = \sum_{n=0}^{\infty}r(n) 5^{\left\lfloor \frac{5n}{2} \right\rfloor} p_1 t^n + \sum_{n=1}^{\infty}s(n) 5^{\left\lfloor \frac{5n-1}{2} \right\rfloor} t^n.
\end{align}  We take $U^{(1)}\{f\}$ and have

\begin{align}
&U^{(1)}\{f\}\notag\\ &= \sum_{n=0}^{\infty}r(n) 5^{\left\lfloor \frac{5n}{2} \right\rfloor} U^{(1)}\{p_1 t^n\} + \sum_{n=1}^{\infty}s(n) 5^{\left\lfloor \frac{5n-1}{2} \right\rfloor} U^{(1)}\{t^n\}\\
&=\sum_{n=0}^{\infty}r(n) 5^{\left\lfloor \frac{5n}{2} \right\rfloor} \Bigg(\sum_{m\ge \left\lceil \frac{n+1}{5} \right\rceil} d_0(m,n) 5^{\left\lfloor \frac{5m-n-1}{2} \right\rfloor} t^m\notag\\ &\ \ \ \ \ \ \ \ \ \ \ \ \ \ \ \ \ \ \ \ \ \ \ \ \ \ \ \ \ \ \ \ \ \ \ \ + \sum_{m\ge \left\lceil \frac{n-1}{5} \right\rceil} d_1(m,n) 5^{\left\lfloor \frac{5m-n+2}{2} \right\rfloor} p_0 t^m\Bigg)\notag\\
&+\sum_{n=1}^{\infty}s(n) 5^{\left\lfloor \frac{5n-1}{2} \right\rfloor} \Bigg(\sum_{m\ge \left\lceil \frac{n}{5} \right\rceil} c(m,n) 5^{\left\lfloor \frac{5m-n-1}{2} \right\rfloor} t^m\Bigg)\\
&=p_0\sum_{m\ge 0}\sum_{n\ge 0}r(n) d_1(m,n)  5^{\left\lfloor \frac{5n}{2} \right\rfloor + \left\lfloor \frac{5m-n+2}{2} \right\rfloor} t^m\label{u1a}\\
&+\sum_{m\ge 1}\sum_{n\ge 0}r(n) d_0(m,n)  5^{\left\lfloor \frac{5n}{2} \right\rfloor + \left\lfloor \frac{5m-n-1}{2} \right\rfloor} t^m\label{u1b}\\
&+\sum_{m\ge 1}\sum_{n\ge 1}s(n) c(m,n)  5^{\left\lfloor \frac{5n-1}{2} \right\rfloor + \left\lfloor \frac{5m-n-1}{2} \right\rfloor} t^m\label{u1c}.
\end{align}  Examining our power of $5$ for line (\ref{u1a}), noting that $m,n\ge 0$, we find that 

\begin{align}
\left\lfloor \frac{5n}{2} \right\rfloor + \left\lfloor \frac{5m-n+2}{2} \right\rfloor &= \left\lfloor \frac{3n}{2} \right\rfloor + \left\lfloor \frac{5m+n+2}{2} \right\rfloor\notag\\ &\ge \left\lfloor \frac{5m+2}{2} \right\rfloor = \left\lfloor \frac{5m}{2} \right\rfloor + 1.
\end{align}  That is, the coefficient of $p_0t^m$ contains \textit{at least one additional power of 5 more than necessary}.  Similarly, we consider line (\ref{u1b}), with $m\ge 1, n\ge 0$:

\begin{align}
\left\lfloor \frac{5n}{2} \right\rfloor + \left\lfloor \frac{5m-n-1}{2} \right\rfloor &\ge \left\lfloor \frac{5m+n-1}{2} \right\rfloor \ge \left\lfloor \frac{5m-1}{2} \right\rfloor = \left\lfloor \frac{5m-3}{2} \right\rfloor + 1.
\end{align}  Finally, for line (\ref{u1c}), with $m,n\ge 1$:

\begin{align}
\left\lfloor \frac{5n-1}{2} \right\rfloor + \left\lfloor \frac{5m-n-1}{2} \right\rfloor &= \left\lfloor \frac{3n-1}{2} \right\rfloor + \left\lfloor \frac{5m+n-1}{2} \right\rfloor\\ &\ge 1 + \left\lfloor \frac{5m-1}{2} \right\rfloor > \left\lfloor \frac{5m-3}{2} \right\rfloor + 1.
\end{align}  In both cases, the coefficients of $t^m$ contain at least one additional power of $5$.

We therefore have $U^{(1)}\{f\} = 5\cdot g$, for some $g\in X^{(0)}$.

\end{proof}

We can now prove a slightly stronger version of Theorem 5.

\begin{theorem}
For every $n\in\mathbb{Z}_{>0}$, there exist functions $g_{2n-1}\in X^{(1)}$ and\\ $g_{2n}\in X^{(0)}$ such that 

\begin{align}
L_{2n-1} &= 5^{n-1}g_{2n-1},\text{ and } L_{2n} = 5^{n}g_{2n}.
\end{align}
\end{theorem}

\begin{proof}
Since $L_1 = p_1 \in X^{(1)}$, we have 

\begin{align}
L_2 = U^{(1)}\{L_1\} = U^{(1)}\{p_1\} =5 g_{1},
\end{align} with $g_{1}\in X^{(0)}$.  Suppose that for some $k\in\mathbb{Z}_{>0}$, we have $L_{2k} = 5^k g_{2k}$, with $g_{2k}\in X^{(0)}$.  Then we have 

\begin{align}
L_{2k+1} &= U^{(0)}\{L_{2k}\} = U^{(0)}\{5^k g_{2k}\} = 5^k U^{(0)}\{g_{2k}\} = 5^k g_{2k+1},
\end{align} with $g_{2k+1}\in X^{(1)}$.  Finally, we have

\begin{align}
L_{2k+2} &= U^{(1)}\{5^k g_{2k+1}\} = 5^k U^{(1)}\{g_{2k+1}\} = 5^k\cdot 5\cdot g_{2k+2} = 5^{k+1}g_{2k+2},
\end{align} with $g_{2k+2}\in X^{(0)}$.

By induction, for every $n\in\mathbb{Z}_{>0}$, there must exist a $g_{2n}\in X^{(0)}$ such that $L_{2n} = 5^n g_{2n}$.

Since for every $n\in\mathbb{Z}_{>0}$,
\begin{align}
L_{2n+1} = U^{(0)}\{5^n g_{2n}\} = 5^n U^{(0)}\{g_{2n}\},
\end{align} and since $L_1 = 5^0 p_1,$ we immediately derive that there must exist a $g_{2n-1}\in X^{(1)}$ such that 
\begin{align}
L_{2n-1} = 5^{n-1} g_{2n-1}.
\end{align}

\end{proof}

\begin{corollary}
For every $n\in\mathbb{Z}_{>0}$, $L_{2n} \equiv 0\pmod{5^n}$.
\end{corollary}

\begin{proof}
For every $n\in\mathbb{Z}_{>0}$, $L_{2n} = 5^n g_{2n}$ for some $g_{2n}\in X^{(0)}$.  And the functions of $X^{(0)}$ have integer coefficients.
\end{proof}

With this, we have proven Theorem 5 and Theorem 2.

\section{Initial Cases}

We will now justify the twenty initial relations needed to prove Lemma 4.

In practice, each of these relations was found by using an ansatz, i.e., by guessing.  However, the actual verification of these relations can be achieved through the theory of modular functions.  Given the intricacy of this subject, we can only provide a brief outline here.  The interested reader is invited to consult \cite{Diamond}, \cite[Chapters 1, 2]{Knopp}, and \cite{Newman} for an outline of the general theory, and \cite[Chapters 3--8]{Knopp} \cite{Radu} and \cite{Radu2} for specific applications of the theory.

\subsection{Preliminaries}

We will denote $\mathbb{H}$ as the upper half complex plane, and $\mathrm{SL}(2,\mathbb{Z})$ to be the set of all $2\times 2$ integer matrices with determinant 1.  Furthermore, we let 

\begin{align*}
\mathrm{SL}(2,\mathbb{Z})_{\infty} = \Bigg\{ \begin{pmatrix}
  1 & b \\
  0 & 1 
 \end{pmatrix}\in \mathrm{SL}(2,\mathbb{Z}) \Bigg\}.
\end{align*}  For any given $N\in\mathbb{Z}_{>0}$, let

\begin{align*}
\Gamma_0(N) = \Bigg\{ \begin{pmatrix}
  a & b \\
  Nc & d 
 \end{pmatrix}\in \mathrm{SL}(2,\mathbb{Z}) \Bigg\}.
\end{align*}

Now, the quotient group $\mathrm{SL}(2,\mathbb{Z})/\mathrm{SL}(2,\mathbb{Z})_{\infty}$ is the set

\begin{align*}
\mathrm{SL}(2,\mathbb{Z})/\mathrm{SL}(2,\mathbb{Z})_{\infty} = \left\{ \begin{pmatrix}
  a & b \\
  c & d 
 \end{pmatrix}\mathrm{SL}(2,\mathbb{Z})_{\infty} : a\in\mathbb{Z},\ c\in\mathbb{Z}_{\ge 0},\ \gcd(a,c)=1 \right\}.
\end{align*}  Assuming first that $c\neq 0$, each member of the coset 

\begin{align*}
\begin{pmatrix}
  a & b \\
  c & d 
 \end{pmatrix}\mathrm{SL}(2,\mathbb{Z})_{\infty} = \left\{ \begin{pmatrix}
  a & x \\
  c & y 
 \end{pmatrix}: x,y\in\mathbb{Z},\ ay-cx =1 \right\}
\end{align*} is a matrix with fixed left-components $a,c$, so that we may represent each such coset with the rational number $a/c$.  If we also identify $\infty$ with the expression $1/0$, then we have a bijection between $\mathbb{Q}\cup\{\infty\}$ and $\mathrm{SL}(2,\mathbb{Z})/\mathrm{SL}(2,\mathbb{Z})_{\infty}$.

Moreover, because $\Gamma_0(N)\subseteq \mathrm{SL}(2,\mathbb{Z})$ is a finite-index subgroup \cite[Chapter 1, Section 1.2]{Diamond}, there are only a finite number of distinct cosets corresponding to

\begin{align*}
\mathrm{SL}(2,\mathbb{Z})/\Gamma_0(N).
\end{align*}  We can therefore partition $\mathbb{Q}\cup\{\infty\}$ into a finite number of sets, each corresponding to a double coset of

\begin{align*}
\Gamma_0(N)\backslash \mathrm{SL}(2,\mathbb{Z})/\mathrm{SL}(2,\mathbb{Z})_{\infty}.
\end{align*}  Each double coset is referred to as the cusp represented by $a/c$, in the orbit space defined by the action of $\Gamma_0(N)$ on $\mathbb{H}\cup \mathbb{Q}\cup \{\infty\}$ (This orbit space is the corresponding modular curve to $\Gamma_0(N)$.  See \cite[Chapters 2, 3]{Diamond} for a thorough treatment on the geometrical interpretation of $\Gamma_0(N)$ over $\mathbb{H}$ and the nature of its cusps).

\begin{definition}
Let $q=e^{2\pi i\tau}$, with $\tau\in\mathbb{H}$.  A function $f:\mathbb{H}\rightarrow\mathbb{C}$ is modular with respect to $\Gamma_0(N)$ if the following three conditions apply:

\begin{enumerate}
\item $f(\tau)$ is holomorphic for all $\tau\in\mathbb{H},$
\item \begin{align*}
f\left( \frac{a\tau+b}{Nc\tau+d} \right) = f(\tau),\ \text{for all } \begin{pmatrix}
  a & b \\
  Nc & d 
 \end{pmatrix}\in\Gamma_0(N),
\end{align*}
\item \begin{align*}
f\left( \frac{a\tau+b}{c\tau+d} \right) = \sum_{m=m_{\gamma}(f)}^{\infty}\alpha_{\gamma}(m)q^{m \gcd(c^2,N)/ N},\ \text{for all } \gamma=\begin{pmatrix}
  a & b \\
  c & d 
 \end{pmatrix}\in \mathrm{SL}(2,\mathbb{Z}),
\end{align*} with $m_{\gamma}(f)\in\mathbb{Z}$, and $\alpha_{\gamma}(m)\in\mathbb{C}$ for all $m\ge m_{\gamma}(f)$.
\end{enumerate}  Here, we refer to $m_{\gamma}(f)$ as the order of $f$ at the cusp represented by $a/c$, respectively by $\gamma=\begin{pmatrix}
  a & b \\
  c & d 
 \end{pmatrix}\in \mathrm{SL}(2,\mathbb{Z})$, over $\Gamma_0(N)$.
\end{definition}

It can be proved \cite[Section 1, Lemma 2]{Radu} that if $\gamma_1, \gamma_2\in \mathrm{SL}(2,\mathbb{Z})$, with 

\begin{align*}
\gamma_j = \begin{pmatrix}
  a_j & b_j \\
  c_j & d_j 
 \end{pmatrix},
\end{align*} such that $\gamma_1\in\Gamma_0(N)\gamma_2 \mathrm{SL}(2,\mathbb{Z})_{\infty}$, then $m_{\gamma_1}(f) = m_{\gamma_2}(f)$.  This fact ensures that any modular function $f$ has a unique order at each cusp of $\Gamma_0(N)$.  Finally, because we may represent each cusp by a member of $\mathbb{Q}\cup \{\infty\}$, we may write 
\begin{align*}
m_{\gamma} = m_{a/c},
\end{align*} with $a,c$ the left-components of $\gamma$.

We now define the relevant sets of all modular functions:

\begin{definition}
Let $\mathcal{K}(N)$ be the set of all modular functions over $\Gamma_0(N)$, and $\mathcal{K}^{\infty}(N)\subset \mathcal{K}(N)$ to be those modular functions over $\Gamma_0(N)$ with a pole only at the cusp at $\infty$ (the cusp that can be represented with $1/N$, or equivalently, $1/0$).  These are both commutative rings with 1, and standard addition and multiplication \cite[Section 2.1]{Radu}.
\end{definition}

We now give three key theorems that will prove useful in checking the modularity of certain functions.  The first is a theorem by Newman \cite[Theorem 1]{Newman}:

\begin{theorem}
Let $f = \prod_{\delta | N} \eta(\delta\tau)^{r_{\delta}}$, with $\hat{r} = (r_{\delta})_{\delta | N}$ an integer-valued vector, for some $N\in\mathbb{Z}_{>0}$.  Then $f$ is a modular function over $\Gamma_0(N)$ if and only if the following apply:

\begin{enumerate}
\item $\sum_{\delta | N} r_{\delta} = 0;$
\item $\sum_{\delta | N} \delta r_{\delta} \equiv 0\pmod{24};$
\item $\sum_{\delta | N} \frac{N}{\delta}r_{\delta} \equiv 0\pmod{24};$
\item $\prod_{\delta | N} \delta^{|r_{\delta}|}$ is a perfect square.
\end{enumerate}

\end{theorem}

To study the order of an eta quotient at a given cusp, we make use of a theorem that can be found in \cite[Theorem 23]{Radu}, generally attributed to Ligozat:

\begin{theorem}
If $f = \prod_{\delta | N} \eta(\delta\tau)^{r_{\delta}}$ is a modular function over $\Gamma_0(N)$, then the order of $f$ at the cusp represented by $a/c$ is given by the following:

\begin{align*}
m_{a/c}(f) = \frac{N}{24\gcd{(c^2,N)}}\sum_{\delta | N} r_{\delta}\frac{\gcd{(c,\delta)}^2}{\delta}.
\end{align*}

\end{theorem}

Our last, and possibly most important, theorem, is \cite[Chapter 2, Theorem 7]{Knopp}:

\begin{theorem}
For a given $N\in\mathbb{Z}_{>0}$, if $f\in\mathcal{K}(N)$ has no poles at any cusp of $\Gamma_0(N)$, then $f$ must be a constant.
\end{theorem}

This is immensely useful for verifying that two modular functions over the same space are equivalent.  If $f,g\in\mathcal{K}^{\infty}(N)$, and their principal parts match, then $f-g\in\mathcal{K}(N)$ can have no poles at any cusp.  This forces $f-g$ to be a constant.  If their constants also match, then $f-g=0$, i.e., $f=g$.

\subsection{Computing the Initial Cases}

We now apply the machinery of modular functions to $\sigma, \mu, \rho$ over $\Gamma_0(20)$.  We begin by making use of \cite[Lemma 5.3]{Radu2} to derive a set of representatives for the distinct cusps of $\Gamma_0(20)$.  Doing so gives us the following:

\begin{align}
\left\{ \frac{1}{20}, \frac{1}{10}, \frac{1}{5}, \frac{1}{4}, \frac{1}{2}, 1 \right\}.\label{cusps20}
\end{align}

Theorems 9 and 10 allow us to quickly verify that $\sigma, \mu, \rho\in\mathcal{K}^{\infty}(20)$.  For instance, in the case of $\sigma$, we have

\begin{align*}
\sigma = \prod_{\delta | 20} \eta(\delta\tau)^{r_{\delta}},
\end{align*} with $\hat{r} = (0,-2,4,0,2,-4)$, as defined in (\ref{sigma}).  Here, we can immediately check the four key conditions of Theorem 9:

\begin{align*}
\sum_{\delta | N} r_{\delta} &= -2 + 4 + 2 - 4 = 0;\\
\sum_{\delta | N} \delta r_{\delta} &= 2(-2) + 4(4) + 10(2) + 20 (-4) = -48 \equiv 0\pmod{24};\\
\sum_{\delta | N} \frac{N}{\delta}r_{\delta} &= 10(-2) + 5(4) + 2(2) + 1(-4) = 0 \equiv 0\pmod{24};\\
\prod_{\delta | N} \delta^{|r_{\delta}|} &= 2^{2}\cdot 4^{4}\cdot 10^{2}\cdot 20^{4} = 2^{16}5^{6} = (32000)^2.
\end{align*}  We can now use Theorem 10 to compute the order of $\sigma$ at each cusp.

\begin{align*}
\text{ord}^{20}_{1/20} &= -2,\\
\text{ord}^{20}_{1/4} &= 2,\\
\text{ord}^{20}_{1/10} &= \text{ord}^{20}_{1/5} = \text{ord}^{20}_{1/2} = \text{ord}^{20}_{1} = 0.
\end{align*}

This verifies that $\sigma\in\mathcal{K}^{\infty}(20)$.  Similarly, $\mu, \rho\in\mathcal{K}^{\infty}(20)$.  Recall the definitions of $p_0, p_1$, in (\ref{p0}) and (\ref{p1}):

\begin{align*}
p_0 =& 31 \rho^{-1} - 22 \sigma \rho^{-1} - 9 \sigma^2 \rho^{-1} - 208 \rho^{-2} - 96 \sigma \rho^{-2} + 304 \sigma^2 \rho^{-2}\notag\\& - 32 \rho^{-1} \mu + 416 \rho^{-2} \mu + 416 \sigma \rho^{-2} \mu - 208 \rho^{-2} \mu^2,\\
p_1 =& 261 \rho^{-1} + 126 \sigma \rho^{-1} + 13 \sigma^2 \rho^{-1} - 960 \rho^{-2} - 5120 \sigma \rho^{-2} - 320 \sigma^2 \rho^{-2}\notag\\& + 
 64 \rho^{-1} \mu + 320 \rho^{-2} \mu - 1280 \sigma \rho^{-2} \mu + 640 \rho^{-2} \mu^2.
\end{align*}

If we multiply both functions through by $\rho^2$, and remember from Definition 7 that $\mathcal{K}^{\infty}(N)$ is closed under addition, subtraction, and multiplication, we have: 

\begin{align}
\rho^2 p_0, \rho^2 p_1\in\mathcal{K}^{\infty}(20).\label{Modpj}
\end{align}  Finally, we can use the same methods to show that 
\begin{align}
\rho^2 t \in\mathcal{K}^{\infty}(20).\label{Modt}
\end{align}

This gives us a direct means to verify the twenty fundamental relations below.  On the left hand side of each relation, we let $-4\le n\le 0$.  With these values of $n$, we can determine that for $j=0,1$, and $l=1,2$,

\begin{align}
U^{(j)}\{\rho^{l} t^n\}, U^{(j)}\{\sigma^l t^n\}, U^{(j)}\{\mu^l t^n\}, U^{(j)}\{t^n\}\in\mathcal{K}(20),\label{build}
\end{align} with the use of Radu's algorithm \cite[Section 3.1]{Radu}.  We give an example below.  

Since Lemma 1 and (\ref{linU0}) show that $U^{(j)}$ is linear, we can therefore demonstrate that 

\begin{align}
U^{(0)}\{p_0 t^n\}, U^{(1)}\{p_1 t^n\}\in\mathcal{K}(20).
\end{align}  Theorem 10 can quickly be used to check that $\rho$ has negative order only at infinity; it has positive order at every other cusp except $1/5$.  However, the functions in (\ref{build}) do not have negative order at $1/5$, as can be checked with \cite[Theorem 47]{Radu}.  Therefore, a sufficiently large prefactor of $\rho^{k}$ can then be used to push each function in (\ref{build}), and therefore $U^{(j)}\{p_j t^n\}$, to $\mathcal{K}^{\infty}(20)$.

We now take advantage of Theorem 11.  Because both sides of each of the twenty relations below is a member of $\mathcal{K}(20)$, and a sufficiently large power of $\rho$ can put both sides into $\mathcal{K}^{\infty}(20)$, verification of each relation is merely a matter of comparing the principal parts at infinity of each side---a finite task that can easily be done by computer.

A mild exception holds for the final relation of Group IV.  Since $t^{-1}$ has a pole at the cusp $1/5$, no prefactor of $\rho$ is sufficient to push it into $\mathcal{K}^{\infty}(20)$.  We need to multiply both sides of each relation by $t$, as well as a sufficient power of $\rho$.

As an example, we choose the second relation of Group I.  In computing 

\begin{align}
U^{(0)}\{t^{-1}\} &= U_5\left\{ A(q)\cdot\frac{(q;q)_{\infty}^6}{q(q^5;q^5)_{\infty}^6} \right\}\\
&= U_5\left\{ \frac{(q;q)^3_{\infty}(q^{2};q^{2})^5_{\infty}(q^{25};q^{25})^3_{\infty}(q^{100};q^{100})^2_{\infty}}{(q^{4};q^{4})^2_{\infty}(q^5;q^5)_{\infty}^6(q^{50};q^{50})^5_{\infty}}\right\}\\
&= U_5\left\{ \prod_{\delta | 100} (q^{\delta};q^{\delta})_{\infty}^{r_{\delta}} \right\},
\end{align} with $\hat{r}=(r_{\delta})_{\delta | 100} = (3,5,-2,-6,0,0,3,-5,2)$.  

Taking $\hat{s}=(s_{\delta})_{\delta | 20} = (4,0,4,-4,16,-20)$, as the doubled powers of the factors of $\rho$ (i.e. $\rho^{2}$), we find that $(5,100,20,0,\hat{r})\in\Delta^{\ast}$ \cite[Definition 35]{Radu}.  As a result, we can verify that $\hat{r}$ and $\hat{s}$ satisfy the equations of \cite[Theorem 45]{Radu}, with $v=0$ and $|P_{5,\hat{r}}(0)|=1$, therefore guaranteeing that

\begin{align}
\rho^{2}U^{(0)}\{t^{-1}\} = \rho^{2}\cdot U_5\left\{ \prod_{\delta | 100} (q^{\delta};q^{\delta})_{\infty}^{r_{\delta}} \right\} \in\mathcal{K}(20),
\end{align} and \cite[Theorem 47]{Radu} that

\begin{align}
\text{ord}_{\gamma}^{20}\left(\rho^{2}\cdot U_5\left\{ \prod_{\delta | 100} (q^{\delta};q^{\delta})_{\infty}^{r_{\delta}} \right\}\right) \ge 0,
\end{align} for every $\gamma\in \mathrm{SL}(2,\mathbb{Z}) \backslash \Gamma_0(20)$.  That is to say, $\rho^{2}U^{(0)}\{t^{-1}\}$ contains no poles except at the cusp of $\infty$.

We have therefore verified that

\begin{align}
\rho^{2}U^{(0)} \{t^{-1}\}\in\mathcal{K}^{\infty}(20).\label{u0tcheck1}
\end{align}

Since (\ref{Modpj}) and (\ref{Modt}) imply that

\begin{align}
\rho^{2} (1 + 5^2t - 5p_1)\in\mathcal{K}^{\infty}(20),\label{u0tcheck2}
\end{align} we need only compare the principal parts and the constants of (\ref{u0tcheck1}) and (\ref{u0tcheck2}).  We find that both expressions have the identical principal part and constant

\begin{align}
&\frac{1}{q^{10}} - \frac{44}{q^9} - \frac{138}{q^8} - \frac{372}{q^7} - \frac{989}{q^6} - \frac{1584}{q^5}\notag\\ &- \frac{2814}{q^4} - \frac{4356}{q^3} - \frac{5897}{q^2} - \frac{9508}{q} - 12696.
\end{align}  As a result, both expressions must be equal: 

\begin{align}
\rho^{2}U^{(0)} \{t^{-1}\} &= \rho^{2} (1 + 5^2t - 5p_1),\\
U^{(0)} \{t^{-1}\} &= 1 + 5^2t - 5p_1.
\end{align}

\subsection{Group I}

\begin{align}
U^{(0)} \{1\} &= p_1,\\
U^{(0)} \{t^{-1}\} &= 1  + 5^2 t - 5 p_1,\\
U^{(0)} \{t^{-2}\} &= -9  + 5^5 t^2 + 9\cdot 5 p_1,\\
U^{(0)} \{t^{-3}\} &=17\cdot 5  + 5^8 t^3 - 17\cdot 5^2 p_1,\\
U^{(0)} \{t^{-4}\} &=-161\cdot 5 + 5^{11} t^4 + 161\cdot 5^2 p_1.
\end{align}

\subsection{Group II}
\begin{align}
U^{(0)} \{p_0\} =&-63\cdot 5^2 t - 104\cdot 5^5 t^2 - 189\cdot 5^7 t^3 - 24\cdot 5^{10} t^4 - 5^{13} t^5\notag\\& + p_1\big(1 - 63\cdot 5^2 t - 104\cdot 5^5 t^2 - 189\cdot 5^7 t^3 - 24\cdot 5^{10} t^4\notag\\& - 5^{13} t^5\big),\\
U^{(0)} \{p_0t^{-1}\} =& 5^2 t -6 p_1,\\
U^{(0)} \{p_0t^{-2}\} =& -9  - 5^3 t  + 5^5 t^2 + p_1( 9\cdot 5 - 5^3 t),\\
U^{(0)} \{p_0t^{-3}\} =& 17\cdot 5 - 5^6 t^2 + 5^8 t^3 - p_1(17\cdot 5^2 - 5^6 t^2),\\
U^{(0)} \{p_0t^{-4}\} =&-161\cdot 5 - 5^9 t^3 + 5^{11} t^4 + p_1(161\cdot 5^2 - 5^9 t^3).
\end{align}

\subsection{Group III}

\begin{align}
U^{(1)} \{1\} &= 1,\\
U^{(1)} \{t^{-1}\} &= -6-5^2t,\\
U^{(1)} \{t^{-2}\} &=54-5^5t^2,\\
U^{(1)} \{t^{-3}\} &=-102\cdot 5-5^8t^3,\\
U^{(1)} \{t^{-4}\} &=966\cdot 5-5^{11}t^4.
\end{align}

\subsection{Group IV}
\begin{align}
U^{(1)} \{p_1\} =& 233\cdot 5^2 t + 1188\cdot 5^4 t^2 + 317\cdot 5^7 t^3 + 31\cdot 5^{10} t^4 + 5^{13} t^5\notag\\&+p_0(2\cdot 5 + 44\cdot 5^3 t + 14\cdot 5^6 t^2 + 5^9 t^3),\\
U^{(1)} \{p_1t^{-1}\} =& 13 + 5^2 t + 5p_0,\\
U^{(1)} \{p_1t^{-2}\} =& -66 - 5^4 t + 5^5 t^2 + 5^4tp_0,\\
U^{(1)} \{p_1t^{-3}\} =& 114\cdot 5 - 5^7 t^2 + 5^8 t^3+5^7t^2p_0,\\
U^{(1)} \{p_1t^{-4}\} =& -1037\cdot 5 + 82\cdot 5^4 t + 112\cdot 5^6 t^2 - 7\cdot 5^9 t^3 - 4\cdot 5^{11} t^4\notag\\&+p_0\left(t^{-1}-2\cdot 5^3-44\cdot 5^5t-14\cdot 5^8t^2-4\cdot 5^{10}t^3\right).
\end{align}

\section{Acknowledgments}

This research was funded by the Austrian Science Fund (FWF): W1214-N15, project DK6.

I am extremely grateful to Youn-Seo Choi, Byungchan Kim, and Jeremy Lovejoy for their work leading to the conjecture that was the subject of this paper.  Finally, my thanks to Professor Peter Paule and Dr. Cristian-Silviu Radu for their insights and mentorship.


\begin{thebibliography}{X}

\bibitem{Andrews} G.E. Andrews, \textit{The Theory of Partitions}, Encyclopedia of Mathematics and its Applications, Vol. 2, Addison-Wesley, 1976. Reissued, Cambridge, 1998.

\bibitem{Atkin} A.O.L. Atkin, ``Proof of a Conjecture of Ramanujan," \textit{Glasgow Mathematical Journal}, 8 (1), pp. 14-32, 1967.

\bibitem{CKL} Y. Choi, B. Kim, J. Lovejoy, ``Overpartitions into Distinct Parts Without Short Sequences," \textit{Journal of Number Theory}, 175, pp. 117-133, 2017.

\bibitem{Diamond} F. Diamond, J. Shurman, \textit{A First Course in Modular Forms}, 4th Printing., Springer, 2016.

\bibitem{Euler} L. Euler, \textit{Introductio in Analysin Infinitorum}, Chapter 16.  Marcum--Michaelem Bousquet, Lausannae, 1748.

\bibitem{Knopp} M. Knopp, \textit{Modular Functions in Analytic Number Theory}, 2nd Ed., AMS Chelsea Publishing, 1993.

\bibitem{Kolitsch} L.W. Kolitsch, ``Rogers--Ramanujan Subpartitions and Their Connections to Other Partitions," \textit{Ramanujan Journal}, 16, pp. 163-167, 2008.

\bibitem{Newman}  M. Newman, ``Construction and Application of a Class of Modular Functions (II)," \textit{Proc. London Math. Soc.}, 3(9), 1959.

\bibitem{Paule}  P. Paule, S. Radu, ``The Andrews--Sellers Family of Partition Congruences," \textit{Advances in Mathematics}, pp. 819-838, 2012.

\bibitem{Radu} S. Radu, ``An Algorithmic Approach to Ramanujan--Kolberg Identities," \textit{Journal of Symbolic Computation}, 68 (1), pp. 225-253, 2015.

\bibitem{Radu2} S. Radu, ``An Algorithm to Prove Algebraic Relations Involving Eta Quotients," \textit{Annals of Combinatorics}, 22(2), pp. 377-391, 2018.

\bibitem{Ramanujan} S. Ramanujan, ``Some Properties of $p(n)$, the Number of Partitions of $n$", \textit{Proceedings of the Cambridge Philosophical Society}, 29, pp. 207-210, 1919.

\bibitem{Watson} G.N. Watson, ``Ramanujans Vermutung über Zerf\"allungsanzahlen," \textit{J. Reine Angew. Math.,} 179, pp. 97-118, 1938.

\end{thebibliography}
\end{document}